\documentclass[10pt,twoside]{amsart}

\usepackage{amscd}
\usepackage{amssymb} 
\usepackage{verbatim}
\usepackage{amsmath}
\usepackage{amsxtra}
\usepackage{esint} 
\usepackage[noadjust]{cite}         
\usepackage{color}

\newtheorem{theorem}{Theorem}[section]
\newtheorem{lemma}[theorem]{Lemma}

\newtheorem{proposition}[theorem]{Proposition}

\theoremstyle{definition}

\newtheorem{remark}[theorem]{Remark}

\numberwithin{equation}{section}

\begin{document}

\title[Existence and non-existence for a fractional problem]{Existence and non-existence results for a semilinear fractional Neumann problem}

\author[E. Cinti]{Eleonora Cinti}
\address{Dipartimento di Matematica, Alma Mater Studiorum Universit\`a di Bologna,
  Piazza di Porta San Donato 5, 40126 Bologna, Italy} 
  \email{eleonora.cinti5@unibo.it}  


\author[F. Colasuonno]{Francesca Colasuonno*}
  \email{francesca.colasuonno@unibo.it}

\begin{abstract}
We establish a priori $L^\infty$-estimates for non-negative solutions of a semilinear nonlocal Neumann problem. As a consequence of these estimates, we get non-existence of non-constant solutions under suitable assumptions on the diffusion coefficient and on the nonlinearity. Moreover, we prove an existence result for radial, radially non-decreasing solutions in the case of a possible supercritical nonlinearity, extending to the case $0<s\le 1/2$ the analysis started in \cite{CC}.
\end{abstract}

\thanks{* Corresponding author.}


\maketitle


\section{Introduction} \label{Intro}
For $s\in(0,1)$, we consider the following nonlocal Neumann problem 
\begin{equation}\label{P}
\begin{cases}
d(-\Delta)^s u+u=u^{q-1}\quad&\mbox{in }B,\\
u\ge 0\quad&\mbox{in }B,\\
\mathcal N_s u=0\quad&\mbox{in }\mathbb R^n\setminus \overline B,
\end{cases}
\end{equation}
where $d>0$, $B$ is the unit ball of $\mathbb R^n$ with $n\ge1$, $q>2$, $(-\Delta)^s$ denotes the fractional Laplacian
\begin{equation}\label{FL}
(-\Delta)^su(x):=c_{n,s}\,\mathrm{PV}\int_{\mathbb R^n}\frac{u(x)-u(y)}{|x-y|^{n+2s}}dy,
\end{equation}
$c_{n,s}$ is a normalization constant, and $\mathcal N_s$ is the following nonlocal normal derivative
\begin{equation}\label{Neu}
\mathcal N_s u(x):=c_{n,s}\int_B\frac{u(x)-u(y)}{|x-y|^{n+2s}}dy\quad\mbox{for all }x\in\mathbb R^n\setminus\overline B,
\end{equation}
first introduced in \cite{DROV}. We observe that such nonlocal Neumann condition makes the structure of problem \eqref{P} variational, we refer to \cite[Sections 1 and 3]{DROV} for further comments on this.

Problem \eqref{P} is the analogue in the nonlocal setting of the Lin-Ni-Takagi problem, studied since the eighties, see for instance \cite{LN, LNT, N, NT}.

In \cite{CC}, the existence of a non-constant solution of \eqref{P} was established in the case $s>1/2$, for a more general nonlinearity $f(u)$. 
The aim of this paper is twofold: on one hand we establish a non-existence result for \eqref{P} for any $s\in(0,1)$, on the other hand we complement the existence analysis performed in \cite{CC} with the case $s\le 1/2$. 
Here, to avoid technicalities, we consider only the prototype nonlinearity $u^{q-1}$, but all the results can be extended to more general $f(u)$ (see Remarks \ref{rem1} and \ref{rem2} below).

In order to present our results, we need to introduce some notation. For the operator $(-\Delta)^{s}$ in $B$ under nonlocal Neumann boundary conditions, we denote by $\lambda_{2}$ the second eigenvalue and by $\lambda^+_{2,\mathrm{r}}$ the second eigenvalue whose corresponding eigenfunction is radial and radially non-decreasing, cf. the definitions in \eqref{lambda2rad+} below. 
Moreover, for every $n\ge 1$, we denote by $2^*_n$ the critical exponent for the fractional Sobolev inequality, namely, 
\[
2_n^*:=\begin{cases}\frac{2n}{n-2s}\quad&\mbox{if }0<s<\frac{n}{2},\\
+\infty&\mbox{if }s\ge\frac{n}{2}.
\end{cases}
\]

Our main result on non-existence can be stated as follows. 

\begin{theorem}\label{thm:main-nonexistence}
Let $s\in(0,1)$, $n\ge 1$, and $2<q<\frac{2^*_n+2}{2}$. There exists $d^*>0$ such that for every $d>d^*$, problem \eqref{P} admits only constant solutions.
\end{theorem}
Theorem \ref{thm:main-nonexistence} extends to the nonlocal setting a result by Ni and Takagi (see \cite[Theorem 3]{NT}). The proof is based on a uniform $L^\infty$- estimate, which we state here below.
	
	\begin{theorem}\label{Linfty-nonex}
		Let $s\in(0,1)$, $n\ge 1$, and $2<q<\frac{2^*_n+2}{2}$. There exists a positive constant $K_\infty=K_\infty(n,s,q)$, depending only on $n,\,s$, and $q$, such that for every $u\in H^s_{B,0}$ solution of \eqref{P}, the following estimate holds true: 
		\begin{equation}\label{eq:Lambda}
			\|u\|_{L^\infty(B)}\le K_\infty\max\left\{1,\frac{1}{d^{\Lambda_q}}\right\},\quad \Lambda_q:=\begin{cases}\frac{2^*_n}{2^*_n-2(q-1)}&\mbox{ if }2^*_n<\infty\\1&\mbox{ otherwise.}\end{cases}
			\end{equation}
		\end{theorem}
The definition of solution and the functional space $H^s_{B,0}$ are the natural ones for problem \eqref{P} and will be formally introduced in Section \ref{S:Prelim}.

\begin{remark}\label{rem1}
		Some comments on Theorem \ref{thm:main-nonexistence} are in order.
\begin{itemize} 
	\item[(i)] The value of $d^*$ can be given explicitly in terms of the $L^\infty$-bound, in the following way:
\begin{equation}
\label{eq:d*}
d^*:=\max\left\{1,\frac{(q-1)K_\infty^{q-2}-1}{\lambda_2}\right\}.
\end{equation}
%
\item[(ii)] In Theorem \ref{thm:main-nonexistence}, we have considered the nonlinearity $u^{q-1}$, just for the sake of simplicity. The same result holds true for every nonlinearity $f(u)$ satisfying (see conditions (A1)-(A3) in \cite{NT}): 
\begin{enumerate}
	\item[1.] $f\in C^1([0,+\infty))$, 
	\item[2.] there exist positive constants $a_0,\,a$, and $2<q_0<q<\frac{2^*_n +2}{2}$, such that $$a_0u^{q_0-1}\le f(u)\le a u^{q-1}\quad \mbox{for } u \, \mbox{sufficiently large}.$$
\end{enumerate}
\end{itemize}
	\end{remark}

We can state now our existence result, which extends to the case $0<s\le  1/2$, a result contained in \cite{CC}. 

\begin{theorem}\label{thm:main}
Let $s\in (0,\frac{1}{2}]$, and assume that $q< \frac{2^*_1+2}{2}$.
If $q> 2+d\lambda^+_{2,\mathrm{r}}$, there exists a non-constant radial, radially non-decreasing solution of \eqref{P} which is strictly positive in $B\setminus\{0\}$. 
\end{theorem}

Observe that in the previous statement, $2^*_1$ represents the critical Sobolev exponent in dimension $1$, and not in dimension $n$. The interest here, as in \cite{CC}, is proving existence of solutions in the case of a supercritical nonlinearity (i.e. $q>2^*_n$).
We remark that the range $\big(2+d\lambda^+_{2,\mathrm{r}},\frac{2(1-s)}{1-2s}\big)$, in which $q$ can be taken, is certainly not empty when $s$ is in a (sufficiently small) left neighborhood of ${1}/{2}$. In particular, in the limit case $s=1/2$, for $q$ large enough all assumptions are satisfied. Instead, for smaller values of $s$ the range could become empty.
	For simplicity we have stated the result in the unit ball, but one can see that the same statement holds true in a ball of radius $R$ (as in \cite{CC}) with exactly the same conditions on $q$ (replacing $\lambda^+_{2,\mathrm{r}}$ with the analogous eigenvalue $\lambda^+_{2,\mathrm{r}}(B_R)$ in the set $B_R$). Hence,
 when $s$ is close to zero, one can look at Theorem~\ref{thm:main} as an existence result in a sufficiently large ball $B_R$, since (by scaling) $\lambda^+_{2,\mathrm{r}}(B_R)\to 0$ as $R\to\infty$.
 
It is worth mentioning also that, in terms of the diffusion coefficient $d$, Theorem~\ref{thm:main} ensures that if $0<d<d^{**}:=(q-2)/\lambda_{2,\mathrm{r}}^+$, problem \eqref{P} admits a non-constant solution. Since $K_\infty \ge 1$ (see Remark \ref{rem:K^inftyge1} below) and $\lambda_2\le \lambda_{2,\mathrm{r}}^+$, it results $d^{**}<[(q-1)K^{q-2}_\infty-1]/\lambda_2$. Hence, having in mind the definition of $d^*$ given in \eqref{eq:d*}, both if $d^*=1$ and if $d^*= [(q-1)K^{q-2}_\infty-1]/\lambda_2$, we can conclude that $d^{**}<d^*$. So far, we are not able to prove neither that $d^*$ is sharp for non-existence nor that $d^{**}$ is sharp for existence. Nevertheless, for the local case, in \cite[Lemma 2.7 and Theorem 3.5]{BGT} it has been proved via bifurcation techniques that for $d$ sufficiently large (still remaining below the threshold for non-existence) all non-constant radial solutions are radially {\it decreasing}. Therefore, we suspect that also in the nonlocal case, $d^{**}$ should not be optimal even for the existence of non-constant radial solutions, since we expect that in the gap $(d^{**},d^*)$ problem \eqref{P} admits radially decreasing solutions. In the local case, we refer to \cite[Corollary~1.5-(ii)]{BCN} for the existence of a radial, radially decreasing solution when $q$ satisfies $q>2+d\lambda_{2,\mathrm{r}}$ in the subcritical regime, and to the Introduction of the same paper for a discussion on the importance of the subcritical growth for deriving a priori estimates for radially decreasing solutions.
We refer also to the Introduction of \cite{BGT} for some conjectures on the appearance of nonradial solutions in the local case.

\begin{remark}\label{rem2}
Also in the statement of Theorem \ref{thm:main}, a more general nonlinearity can be considered. In particular the same result holds true if $u^{q-1}$ is replaced by a function $f(u)$ satisfying the following conditions (see $(f_1)$-$(f_3)$ in \cite{CC}):
	\begin{itemize}
		\item[$(f_1)$] $f'(0)=\lim_{t\to 0^+}\frac{f(t)}{t}\in (-\infty,1)$;
		\item[$(f_2)$] $\liminf_{t\to\infty}\frac{f(t)}{t}>1$;
		\item[$(f_3)$] there exists a constant $u_0>0$ such that $f(u_0)=u_0$ and $f'(u_0)>d\lambda_2^{+,\mathrm{r}}+1$.
	\end{itemize}
\end{remark}

The supercritical nature of the problem prevents a priori the use of variational methods for proving existence. To overcome this issue, following previous papers (cf. \cite{ST,BNW,CC}), we work in the convex cone of radial $H^s_{B,0}$-functions which are non-negative and radially non-decreasing. Working in this cone allows to get $L^\infty$-a priori estimates of the form $\|u\|_{L^\infty(B)}\le K'_\infty$ for the solutions $u$ belonging to $\mathcal C$ (cf. Theorem \ref{thm:Linftybound'} below) which are uniform over a whole class of problems (see \eqref{Pg}) to which, in particular, \eqref{P} belongs. As a consequence, this allows in turn to modify the nonlinearity $t^{q-1}$ at infinity (say for $t> K'_\infty+1$) into a subcritical $g(t)$ for proving existence via a mountain pass-type argument, cf. \cite[Section 4]{CC}. It is worth mentioning that the same existence result can also be proved when the domain is an annulus of $\mathbb R^n$. Even more, in such a case, beyond the non-decreasing solution, the existence of a non-constant radially non-increasing solution can be proved with the same technique, working in the cone of radially non-increasing functions. 

As a first step for proving a priori estimates, one needs to use continuous embeddings for fractional Sobolev spaces in the radial setting. In particular, when $s>1/2$ one immediately has that radial increasing $H^s$-functions belong to $L^\infty$ (see \cite[Lemma 3.1]{CC}) and can prove directly uniform $L^\infty$-a priori estimates for solutions of this type. 
On the other side, for treating the case $s\le 1/2$, we need to combine the continuous embedding in $L^p$ (for any $p$ in the case $s=1/2$, and for any $p\le 2^*_1$ in the case $s<1/2$) with a Moser iteration argument.

Once that existence is proved, it remains to show that the solution found is not constant. To do so, it is possible to argue as in \cite[Section 5]{CC}, where it is proved that the energy of the constant 1 is strictly higher than the energy of the minimax solution found before. 
The need for the extra assumption $q>2+d\lambda^+_{2,\mathrm{r}}$ required in the statement of Theorem \ref{thm:main} arises exactly at this stage, in the energy comparison.  

The paper is organized as follows. In Section \ref{S:Prelim}, we recall the functional setting, the definition of solution, and some fractional embedding results. In Section \ref{S:est}, we prove $L^\infty$-a priori estimates for solutions of \eqref{P} and the non-existence result stated in Theorem \ref{thm:main-nonexistence}. Finally, Section \ref{S:existence} is devoted to the proof of the existence of a non-constant radial solution of \eqref{P} for $s\le 1/2$, as stated in Theorem \ref{thm:main}.

\section{Preliminaries}\label{S:Prelim}
In this section we define the functional setting of our problem and recall Sobolev's embedding theorems for fractional spaces, which will be useful in the next sections. 

Let $u:\mathbb R^n\to \mathbb R$ be a measurable function. We introduce the seminorm
\begin{equation}\label{seminorm}
[u]_{H^s_{B,0}}:= \left(\frac{c_{n,s}}{2}\iint_{\mathbb R^{2n}\setminus(B^c)^2}\frac{|u(x)-u(y)|^2}{|x-y|^{n+2s}}dx\,dy\right)^{\frac{1}{2}},
\end{equation}
where $B^c$ denotes the complement of $B$. The space
\[
H^s_{B,0}:=\left\{u:\mathbb R^n\to \mathbb R,\,u \in L^2(B)\, :\, [u]_{H^s_{B,0}}<\infty \right\}
\]
is a  Hilbert space equipped with the norm $\|u\|_{H^s_{B,0}}:=\left(\|u\|_{L^2(B)}^2+[u]_{H^s_{B,0}}^2\right)^{\frac{1}{2}}$.

As mentioned in the Introduction, the Neumann condition \eqref{Neu} is the natural notion to have a variational structure in this setting, indeed, as proven in \cite[Section~3]{DROV}, the following integration by parts formula holds true for all bounded $C^2(\mathbb R^n)$-functions $u,\,v$ :
\begin{equation}\label{eq:int-by-parts}
\begin{split}&\frac{c_{n,s}}{2}\iint_{\mathbb{R}^{2n}\setminus(B^c)^2}\frac{(u(x)-u(y))(v(x)-v(y))}{|x-y|^{n+2s}}\,dx\,dy\\
	&\hspace{6em} = \int_B v(-\Delta)^s u\,dx + \int_{B^c} v \mathcal N_s u\,dx.
\end{split}
\end{equation}

We recall the formal definition of weak solution for problems of the form
\begin{equation}\label{P-generico}
\begin{cases}
d(-\Delta)^s u+u=f(u)\quad&\mbox{in }B,\\
u\ge 0\quad&\mbox{in }B,\\
\mathcal N_s u=0\quad&\mbox{in }\mathbb R^n\setminus \overline B.
\end{cases}
\end{equation}
A non-negative function $u\in H^s_{B,0}$ is a \textit{weak solution} of \eqref{P-generico} if for every $v\in H^s_{B,0}$ it holds
\begin{equation}\label{eq:distr-id}
d\frac{c_{n,s}}{2} \iint_{\mathbb R^{2n}\setminus(B^c)^2}\frac{(u(x)-u(y))(v(x)-v(y))}{|x-y|^{n+2s}}\,dx\,dy + \int_B u\, v\,dx=\int_B  f(u)\,v\, dx.
\end{equation}
Clearly, this definition makes sense when $\int_B  f(u)\,v \,dx$ is finite for every test function $v$. 
For brevity, throughout the paper, we will omit the term \textit{weak} and we will call \textit{solution} of \eqref{P-generico} any function satisfying the distributional identity \eqref{eq:distr-id}. Moreover, when it will be clear from the context, we will avoid distributional identities/inequalities, using directly their strong form: we will write simply $(-\Delta)^s u\le f$ meaning that the corresponding distributional inequality holds with all non-negative test functions.

As for fractional embeddings, we first recall that the following inequality holds between the usual $H^s$-seminorm and our seminorm $[\,\cdot\,]_{H^s_{B,0}}$, by monotonicity of the integrals with respect to domain inclusion
\[
[u]_{H^s_{B,0}}\ge [u]_{H^{s}(B)}.
\] 
Thus, as an easy consequence of the fractional embeddings $H^s(B)\hookrightarrow L^p(B)$ (see for example Section 7 in \cite{Adams} and remind that $H^s(B)=W^{s,2}(B)$), we have the following result.

\begin{proposition}\label{embedding}
If $n\neq 2s$, the space $H^s_{B,0}$ is continuously embedded in $L^p(B)$ for every $p\in[1,2_n^*]$. If $n= 2s$ the same continuous embedding holds for  $p\in[1,+\infty)$.
\end{proposition}

As already mentioned in the Introduction, to gain compactness in our supercritical setting, we work with radial functions, in the following cone $\mathcal C$ made of functions sharing the same symmetry, monotonicity, and sign conditions as the solution we are going to detect:
\begin{equation}\label{cone}
\mathcal C:=\left\{u\in H^s_{B,0}\,:\, \begin{aligned}& u=u(r)\mbox{ is radial, }u\ge0\mbox{ a.e. in }\mathbb R^n,\\& u(r)\le  u(s) \mbox{ if $0\le r\le s\le 1$} \end{aligned}\right\},
\end{equation}
where, with abuse of notation, we denote by the same letter $u$ the function in $H^s_{B,0}$ and its radial profile, writing $u(|x|)=u(x)$.

Working with radial functions allows to improve the embedding results, increasing the critical exponent up to $2^*_1$, which is critical for the one-dimensional case. 

Throughout the paper, we denote by $B_R$ the ball of radius $R$ centered at the origin and we omit the radius only for denoting the unit ball. 

\begin{lemma}\label{immersion}
Let $u$ be a radial function in $H^s_{B,0}$. Then,  
\begin{itemize}
\item[(i)] if $s=\frac{1}{2}$, $\left\|u\right\|_{L^p(B\setminus B_{1/2})}\le C\|u\|_{H^s_{B,0}}$ for every $p \in [1,\infty)$;
\item[(ii)] if $0<s<\frac{1}{2}$, $\left\|u\right\|_{L^p(B\setminus B_{1/2})}\le C\|u\|_{H^s_{B,0}}$ for every $p\in [1,2^*_1]$,
\end{itemize}
for some positive constant $C$ depending only on $n,\,s$, and $p$.
As a consequence, if $u$ belongs to the cone $\mathcal C$ of non-negative, radial, radially non-decreasing functions 
introduced in \eqref{cone}, the following inequalities hold:
\begin{itemize}
\item[(iii)] if $s=\frac{1}{2}$, $\left\|u\right\|_{L^p(B)}\le C\|u\|_{H^s_{B,0}}$ for every $p \in [1,\infty)$;
\item[(iv)] if $0<s<\frac{1}{2}$, $\left\|u\right\|_{L^p(B)}\le C\|u\|_{H^s_{B,0}}$ for every $p\in [1,2^*_1]$
\end{itemize}
for some positive constant $C$ depending only on $n,\,s$, and $p$.
\end{lemma}
\begin{proof}
The proofs of (i) and (ii) are the same as in \cite[Lemma 4.3]{BSY} with the only natural changes due to the different boundary conditions. We report here only a sketch of the proof of (ii) to highlight the differences. Arguing as in \cite[formula (4.3)]{BSY}, we can estimate
\[
[u]_{H^s_{B,0}}^2\ge \frac{c_{n,s}}{2}\iint_{B^2}\frac{|u(x)-u(y)|^2}{|x-y|^{N+2s}}\,dx dy \ge C\int_{\frac{1}{2}}^1\int_{\frac{1}{2}}^1\frac{|u(\varrho)-u(r)|^2}{|\varrho-r|^{1+2s}}\,d\varrho dr.
\]
Here and in the rest of the proof, $C$ denotes different positive constants whose exact values are not important.
Thus, using the radial symmetry and the Sobolev embedding in dimension 1 in the interval $(1/2,1)$, and denoting by $\mathcal S$ the corresponding best Sobolev constant, we get
\[
\begin{aligned}
\|u\|_{H^s_{B,0}}^2 &\ge C [u]_{H^s(1/2,1)}^2+\frac{\omega_{N-1}}{2^{N-1}}\int_{\frac{1}{2}}^1 u^2\, d\varrho\\ &\ge \frac{C}{\mathcal S^2}\left(\int_{\frac{1}{2}}^1 u^{\frac{2}{1-2s}}\, d\varrho\right)^{1-2s}\ge \frac{C}{\mathcal S^2 \omega_{N-1}^{1-2s}}\|u\|^2_{L^{\frac{2}{1-2s}}(B\setminus B_{1/2})}. 
\end{aligned}
\]
The proofs of (iii) and (iv) then follow immediately using the radial monotonicity of functions in $\mathcal C$, being for every $p$ as in the statement
\[
\int_{B_{1/2}}|u|^p dx\le |B_{1/2}||u(1/2)|^p \le \frac{|B_{1/2}|}{|B\setminus B_{1/2}|} \int_{B\setminus B_{1/2}}|u|^p dx.
\]
\end{proof}

We remark that a similar result holds also for $s>1/2$, cf. \cite[Lemma 3.1]{CC} and \cite[Lemma 4.3]{BSY}.

Finally, for future use, we recall here the definitions of the following types of second eigenvalue of the fractional Neumann Laplacian in $B$. Denoting by 
\[
\begin{gathered}
H^s_\mathrm{r}:=\{u\in H^s_{B,0}\,:\, u\mbox{ radial }\}\quad\mbox{and}\\
H^{s,+}_\mathrm{r}:=\{u\in H^s_{B,0}\,:\, u\mbox{ radial and radially non-decreasing }\},
\end{gathered}
\]
we introduce
\begin{equation}\label{lambda2rad+}
\begin{aligned}
\lambda_2&:=\inf_{v\in H^s_{B,0},\,\int v=0}\frac{[v]^2_{H^s_{B,0}}}{\int_{B}v^2}\\
& \le\lambda_{2,\mathrm{r}}:=\inf_{v\in H^s_\mathrm{r},\,\int v=0}\frac{[v]^2_{H^s_{B,0}}}{\int_{B}v^2}\le \lambda^+_{2,\mathrm{r}}:=\inf_{v\in H^{s,+}_\mathrm{r},\,  \int v=0}\frac{[v]^2_{H^s_{B,0}}}{\int_{B}v^2},
\end{aligned}
\end{equation}
where the inequalities follow easily from the inclusions $H^{s,+}_{\mathrm{r}}\subset H^s_\mathrm{r}\subset H^s_{B,0}$ and, by the direct method of Calculus of Variations, all these infima are achieved.

\section{$L^\infty$-a priori estimates and non-existence in the subcritical regime}\label{S:est}
In this section we prove Theorem \ref{thm:main-nonexistence}. 
We start with an $L^{2^*_n}$-$L^\infty$ estimate, which is the crucial ingredient in the proof of the $L^\infty$-a priori bound.

\begin{lemma}\label{lem:L2*Linfty-n}
	Let $q>2$.
	If $2^*_n<\infty$, then, provided that $q < 2^*_n$, there exists a constant $K_q=K_q(s,n,q)$, which depends only on $s,\,n$, and $q$, such that
	\begin{equation}\label{L2*Linfty-n}
	\begin{gathered}	\|u\|_{L^\infty(B)}\le K_q\max\left\{1, \frac{1}{d}\right\}^{\delta_q} \left(1+\|u\|^{\gamma_q}_{L^{2^*_n}(B)}\right)\\
		\mbox{ with }\quad
		\delta_q:=
	\frac{1}{2^*_n-q}
	\quad\mbox{ and } \quad
	\gamma_q:=\frac{2^*_n-2}{2^*_n-q}
	\end{gathered}
	\end{equation}

	for every $u\in H^s_{B,0}$ solution of \eqref{P}.

	If $2^*_n=\infty$, for every $p>q$ there exists a constant $K'_{q,p}=K_{q,p}(s,n,q,p)$, which depends only on $s,\,n,\,q$, and $p$, such that
	\begin{equation}\label{L2*Linfty-1/2-n}
\begin{gathered}
\|u\|_{L^\infty(B)}\le K_{q,p} \max\left\{1,\frac{1}{{d}}\right\}^{\delta_{q,p}}\left(1+\|u\|^{\gamma_{q,p}}_{L^p(B)}\right)\\\mbox{ with }\quad \delta_{q,p}:=\frac{1}{p-q}	\quad\mbox{ and } \quad
\gamma_{q,p}:=\frac{p-2}{p-q}.
\end{gathered}
	\end{equation}
\end{lemma}

\begin{proof} 
	The proof is based on the Moser iteration technique, in a version that has been adapted to the fractional setting, cf. \cite{BCSS}.
		We first consider the case $2_n^*<\infty$.
		For $T>0$ and $\beta >1$, we define the following differentiable and convex function:
		\begin{equation}
		\label{eq:varphi}
		\varphi(t)=\varphi_{T,\beta}(t)=
		\begin{cases} 0 \quad & t\le 0\\
		t^\beta & 0<t<T \\
		\beta T^{\beta-1}(t-T) + T^\beta & t\ge T.
		\end{cases}
		\end{equation}
		Observe that $\varphi$ is a Lipschitz function with Lipschitz constant $\beta T^{\beta-1}$. In particular, if $u$ is a solution of \eqref{P}, $\varphi(u)$ and $\varphi(u)\varphi'(u)$ are non-negative functions belonging to $H^s_{B,0}$.
		Moreover, observe that, since $\varphi$ is convex, we have (cf. \cite[Proposition 4]{LPPS})
		\begin{equation}\label{Delta-phi}
		(-\Delta)^s \varphi(u) \leq \varphi'(u)(-\Delta)^s u \quad\mbox{and}\quad 
		\mathcal N_s \varphi(u) \leq \varphi'(u)\mathcal N_s u,
		\end{equation}
		where the inequalities are meant in the distributional sense. Hence, being $\varphi(u)\in H^s_{B,0}$, by Sobolev's embedding -Proposition \ref{embedding}- and using the $0$-Neumann condition and the second inequality in \eqref{Delta-phi}, we get 
		\begin{equation}\label{lower}
		\begin{split}
		\int_B \varphi(u) &(d(-\Delta)^s \varphi(u) + \varphi(u))dx \\ 
		&= d[\varphi(u)]^2_{H^s_{B,0}} +  \int_B \varphi(u)^2 dx - d\int_{B^c} \varphi(u)\mathcal N_s(\varphi(u))dx \\
		 &\ge \min\{1,d\} \|\varphi(u)\|^2_{H^s_{B,0}} \ge \frac{\min\{1,d\}}{C^2} \|\varphi(u)\|^2_{L^{2_n^*}(B)}.
		\end{split}
		\end{equation}
		On the other hand, using the first inequality in \eqref{Delta-phi} and that $u$ is a solution of \eqref{P}, we have
		\begin{equation}\label{upper}
		\begin{split}
		\int_B \varphi(u)&(d(-\Delta)^s \varphi(u)+\varphi(u))\, dx \\
		&\le \int_B \varphi(u)\varphi'(u)d(-\Delta)^s u\,dx + \int_B \varphi(u)^2 \,dx \\
		 &\le \int_B \varphi(u)\varphi'(u) u^{q-1}\,dx + \int_B \varphi(u)^2 \,dx.
		\end{split}
		\end{equation}
		Combining together \eqref{lower} and \eqref{upper}, we deduce that
		\[
		\int_B \varphi(u)\varphi'(u) u^{q-1}\,dx + \int_B \varphi(u)^2 \,dx \ge \frac{\min\{1,d\}}{C^2} \|\varphi(u)\|^2_{L^{2_n^*}(B)},
		\]
		whence, using that $t\varphi'(t)\le \beta\varphi(t)$ and H\"older's inequality, we obtain
		\begin{equation}\label{eq:est-varphiu}
		\begin{aligned}
		&\|\varphi(u)\|^2_{L^{2_n^*}(B)}\\
		&\,\le \beta C^2\max\left\{1,\frac{1}{d}\right\}\left(\int_B \varphi(u)^2u^{q-2}\,dx+\int_B \varphi(u)^2 \,dx\right)\\
		&\,\le \beta C^2\max\left\{1,\frac{1}{d}\right\}\max\left\{1,|B|^{\frac{q-2}{2^*_n}}\right\}\left(1+\|u\|_{L^{2^*_n}(B)}^{q-2}\right)\left(\int_B \varphi(u)^{\frac{2^*_n}{b}}\,dx\right)^{\frac{2b}{2^*_n}}\!\!,
		\end{aligned}
		\end{equation}  
		where  $b:=\frac{2^*_n-q+2}{2}>1$. Since $u\in L^{2^*_n}(B)$ by Proposition \ref{embedding}, for $\beta=b$ we are allowed to let $T\to \infty$ in \eqref{eq:est-varphiu}. This leads to
	\[
	\|u\|_{L^{b2_n^*}(B)}\le b^{\frac{1}{2b}} \left[C\max\left\{1,\frac{1}{\sqrt{d}}\right\}\max\left\{1,|B|^{\frac{1}{2}\frac{q-2}{2^*_n}}\right\}\left(1+\|u\|_{L^{2^*_n}(B)}^{\frac{q-2}{2}}\right)\right]^{\frac{1}{b}}\|u\|_{L^{2^*_n}(B)}.
	\]
	We now set $\beta_m:=b\beta_{m-1}=b^m$ for every $m\in\mathbb N$, iterate the argument, and obtain the recurrence formula
	\[
	\|u\|_{L^{\beta_{m}2_n^*}(B)}\le \left(C_q \max\left\{1,\frac{1}{\sqrt{d}}\right\}\right)^{\frac{1}{\beta_m}}\beta_m^{\frac{1}{2\beta_m}}\left(1+\|u\|^{\frac{q-2}{2}}_{L^{2^*_n}(B)}\right)^{\frac{1}{\beta_m}}\|u\|_{L^{\beta_{m-1}2_n^*}(B)},
	\]
	where $C_q:= C\max\left\{1,|B|^{\frac{1}{2}\frac{q-2}{2^*_n}}\right\}$. Therefore, iterating, we get
		\begin{equation}\label{eq:iterated'}
		\begin{split}	&\|u\|_{L^{\beta_{m}2_n^*}(B)}\\
		&\quad\le \prod_{k=1}^m\beta_k^{\frac{1}{2\beta_k}}\left[C_{q}\max\left\{1,\frac{1}{\sqrt{d}}\right\}\left(1+\|u\|^{\frac{q-2}{2}}_{L^{2^*_n}(B)}\right)\right]^{\sum_{k=1}^m\frac{1}{\beta_k}}\|u\|_{L^{2^*_n}(B)}.
		\end{split}
		\end{equation}
		Taking into account that, in the limit as $m\to\infty$, the following limits hold true: $\beta_m\to\infty$, $\sum_{k=1}^\infty\frac{1}{\beta_k}=\frac{2}{2^*_{n}-q}$, and $\prod_{k=1}^\infty\beta_k^{\frac{1}{\beta_k}}=: {C_0}^2<\infty$, we can pass to the limit in \eqref{eq:iterated'} to get
	\[
	\begin{aligned}
		\|u\|_{L^\infty(B)}&=\lim_{m\to\infty}\|u\|_{L^{\beta_{m}2_n^*}(B)}\\
		&\le {\left(C_{q}\max\left\{1,\frac{1}{\sqrt{d}}\right\}\right)}^{\frac{2}{2^*_n-q}}C_0\left(1+\|u\|^{\frac{q-2}{2}}_{L^{2^*_n}(B)}\right)^{\frac{2}{2^*_n-q}}\|u\|_{L^{2^*_n}(B)}\\
		&\le \max\{2,2^{\frac{2}{2^*_n-q}}\}C_{q}^{\frac{2}{2^*_n-q}}\max\left\{1,\frac{1}{{d}}\right\}^{\frac{1}{2_n^*-q}} C_0\left(1+\|u\|^{\frac{q-2}{2^*_n-q}+1}_{L^{2^*_n}(B)}\right).
	\end{aligned}
	\] 
	Putting $K_q:= \max\{2,2^{\frac{2}{2^*_n-q}}\}C_{q}^{\frac{2}{2^*_n-q}}C_0$ we conclude the proof in this case.

	Finally, as for the case $2^*_n=\infty$ (that is $s=n/2$, namely $n=1$ and $s=1/2$), we can repeat the argument above replacing $2^*_n$ with $p$, for every $p>q$. This leads to the desired estimate,
	\[
	\|u\|_{L^\infty(B)}\le \max\{2,2^{\frac{2}{p-q}}\}C_{q,p}^{\frac{2}{p-q}}\max\left\{1,\frac{1}{{d}}\right\}^{\frac{1}{p-q}} C_0\left(1+\|u\|^{\frac{q-2}{p-q}+1}_{L^p(B)}\right),
	\] 
	where $C_{q,p}:=C\max\left\{1,|B|^{\frac{1}{2}\frac{q-2}{p}}\right\}$, putting  $K_{q,p}:=\max\{2,2^{\frac{2}{p-q}}\}C_{q,p}^{\frac{2}{p-q}}C_0$.
\end{proof}

\begin{remark}\label{rem:K^inftyge1}
By the previous lemma, we know that every solution of \eqref{P} is bounded in $B$. Moreover, testing the weak formulation of the equation with $v\equiv 1$ and using the 0-Neumann condition, we get $\int_B u \,dx =\int_B u^{q-1}\,dx$. Being $u\ge0$, this implies that, if $u\not\equiv 0$, either $1-u^{q-2}$ changes sign in $B$ or $u\equiv 1$. Thus, every non-zero solution $u$ has $\|u\|_{L^\infty(B)}\ge 1$.
\end{remark}

\begin{proof}[Proof of Theorem \ref{Linfty-nonex}]
We argue as in \cite[Lemmas 3.4 - 3.5]{CC}. 
Testing the equation with $v\equiv 1$ as in the previous Remark, we get
\begin{equation*}
	\begin{split}
		\int_B u \,dx& =\int_B u^{q-1}\,dx =\int_{B\cap \{u\le 2\}} u^{q-1}\,dx +\int_{B\cap \{u>2\}} u^{q-1}\,dx \\
		&\ge (1+\delta)\int_{B\cap \{u>2\}} u\,dx,
		\end{split}
\end{equation*}
where $\delta= 2^{q-2}-1$. 
Hence, we deduce that 
$$\int_{B\cap \{u\le 2\}} u\,dx \ge \delta \int_{B\cap \{u>2\}} u\,dx,$$
which implies
\begin{equation*}
	\int_{B\cap\{u>2\}}u\,dx \le \frac{1}{\delta}\int_{B\cap\{u\le 2\}} u\,dx \le \frac{2}{\delta}|B|.
\end{equation*}
We thus obtain that
\begin{equation}\label{q-1}
\int_B u^{q-1}\,dx=\int_B u\,dx \le 2|B|\left(1+\frac{1}{\delta}\right)=:K.
\end{equation}

Testing the equation in \eqref{P} with $u$, using the embedding in Proposition \ref{embedding}, and using that $u\in L^\infty(B)$ by Lemma \ref{lem:L2*Linfty-n}, we obtain for $2^*_n<\infty$
\begin{equation*}
\|u\|_{L^{2^*_n}(B)}^2\le C^2 \|u\|^2_{H^s_{B,0}}{\le} C^2{\max\left\{1,\frac{1}{d}\right\}} \int_B u^q\,dx\le C^2 \max\left\{1,\frac{1}{d}\right\} K \|u\|_{L^\infty(B)}. 
\end{equation*}
Hence, by \eqref{L2*Linfty-n}, 
\begin{equation*}
\begin{aligned}
\|u\|_{L^\infty(B)}&\le K_q\max\left\{1,\frac{1}{d}\right\}^{\delta_q}\left(1+\|u\|_{L^{2^*_n}(B)}^{\gamma_q}\right)\\
&\le K_q \max\left\{1,\frac{1}{d}\right\}^{\delta_q} \left( 1+C^{\gamma_q} \max\left\{1,\frac{1}{d}\right\}^{\frac{\gamma_q}{2}}{K}^{\frac{\gamma_q}{2}} \|u\|_{L^\infty(B)}^{\frac{\gamma_q}{2}}\right)\\
&\le K_q \max\left\{1,\frac{1}{d}\right\}^{\delta_q+\frac{\gamma_q}{2}}\max\{1,{K}^{\frac{\gamma_q}{2}}C^{\gamma_q}\} \left( 1+ \|u\|_{L^\infty(B)}^{\frac{\gamma_q}{2}}\right)
\end{aligned}
\end{equation*}
The assumption $q<\frac{2^*_n+2}{2}$ is equivalent to $\frac{\gamma_q}{2}<1$, hence we deduce the $L^\infty$-bound with
\begin{equation}\label{eq:Kinfty}
K_\infty:=\left(2K_q \max\{1,K^{\frac{\gamma_q}{2}}C^{\gamma_q}\}\right)^{\frac{2}{2-\gamma_q}}
\end{equation}
as in the statement.

The conclusion for $s=1/2$ and $n=1$ follows analogously using \eqref{L2*Linfty-1/2-n}, being $\frac{\gamma_{q,p}}{2}<1$ for $p$ large.  
\end{proof}

\begin{remark} 
(Alternative end of proof) The last part of the previous proof can be replaced by the following more direct one: once we get \eqref{q-1}, we can invoke Lemma \ref{lem:L2*Linfty-n} to obtain, e.g. when $2^*_n<\infty$,
\begin{equation*}
\begin{aligned}
\|u\|_{L^\infty(B)}&\le K_q \max\left\{1,\frac{1}{d}\right\}^{\delta_q} \left[1+\left(\int_B u^{2^*_n} \, dx \right)^{\frac{\gamma_q}{2^*_n}}\right]\\
	&\le K_q \max\left\{1,\frac{1}{d}\right\}^{\delta_q}\left[1+\|u\|_{L^\infty(B)}^{\frac{2^*_n-q+1}{2^*_n}\gamma_q}\left(\int_B u^{q-1}\,dx\right)^{\frac{\gamma_q}{2^*_n}}\right]\\
	&\le K_q\max\left\{1,\frac{1}{d}\right\}^{\delta_q}\left(1+ K^{\frac{\gamma_q}{2^*_n}}\|u\|_{L^\infty(B)}^{\frac{2^*_n-q+1}{2^*_n}\gamma_q}\right).
\end{aligned}
\end{equation*}
Since the assumption $q<\frac{2^*_n+2}{2}$ is equivalent to $\frac{2^*_n-q+1}{2^*_n}\gamma_q<1$, we get the desired $L^\infty$-bound.
\medskip \\
We did not choose to reason as in the present remark because the argument in the proof can be applied also to more general nonlinearities as we will need to do in the next section.
\end{remark}

We are now ready to prove the non-existence result. 
\begin{proof}[Proof of Theorem \ref{thm:main-nonexistence}]
This part of the proof is inspired from \cite[Theorem 3]{NT}. Let $u=\fint_B u\, dx + \phi=: u_0+\phi\in H^s_{B,0}$ be a solution of \eqref{P}. We test the equation in \eqref{P} with $\phi$ to have
\begin{equation}
\label{eq:tested-phi}
\begin{aligned}
d\frac{c_{n,s}}{2}&\iint_{\mathbb R^{2n}\setminus(B^c)^2}\frac{(\phi(x)-\phi(y))^2}{|x-y|^{n+2s}}\,dx\,dy + \int_B (u_0+\phi)\phi\,dx\\ 
&=\int_B (u_0+\phi)^{q-1}\phi\,dx = (q-1)\int_B \left(\int_0^1(u_0+ t\phi)^{q-2}\,dt\right)\phi^2\,dx,
\end{aligned}
\end{equation}
where in the last equality we used that $\int_B\phi\,dx = 0$. 
Now, by the $L^\infty$-bound in Theorem \ref{Linfty-nonex}, $(u_0+ t\phi)^{q-2}\le K_\infty^{q-2}\max\left\{1,\frac{1}{d^{(q-2)\Lambda_q}}\right\}$, moreover, by Poincar\'e's inequality (cf. \eqref{lambda2rad+} and \cite[Lemma 3.10 and Theorem 3.11]{DROV}), 
\[
\frac{c_{n,s}}{2}\iint_{\mathbb{R}^{2n}\setminus(B^c)^2}\frac{(\phi(x)-\phi(y))^2}{|x-y|^{n+2s}}\,dx\,dy\ge \lambda_2\int_B\phi^2\,dx.
\]
Inserting the previous inequalities in \eqref{eq:tested-phi}, we have
\begin{equation}\label{eq:limit-q}
(d\lambda_2+1)\int_B\phi^2\,dx\le (q-1)K_\infty^{q-2}\max\left\{1, \frac{1}{d^{(q-2)\Lambda_q}}\right\}\int_B\phi^2\,dx.
\end{equation}
Hence, for $d$ sufficiently large (more precisely, $d>d^*$, where $d^*$ is given in Remark~\ref{rem1}) we deduce that $\phi$ must be identically zero, i.e. $u$ must be constant. This concludes the proof.
\end{proof}

\begin{remark}
\begin{itemize} 
\item[(i)] (Non-existence for small $q$) The previous non-existence result can be also stated in terms of $q$ as follows:\smallskip\\ 
\emph{For any $d>0$, there exists $q^*>2$ (depending on $d$) such that, for every $q\in (2,q^*)$, problem \eqref{P} does not admit non-constant solutions.}\smallskip \\ 
Indeed, in the case $2^*_n<\infty$, in order to explicit \eqref{eq:limit-q} in terms of $q$ in the limit as $q\to 2$, we compute by \eqref{eq:Kinfty}
\[
\begin{split}
\lim_{q\to 2}K_\infty^{q-2}&=\lim_{q\to2}\left(2K_q \max\{1,K^{\frac{\gamma_q}{2}}C^{\gamma_q}\}\right)^{\frac{2(q-2)}{2-\gamma_q}}\\
&=\lim_{q\to 2}(1+o(1))\left(1+\frac{1}{2^{q-2}-1}\right)^{\frac{\gamma_q(q-2)}{2-\gamma_q}}=1,
\end{split}
\]
where we have taken into account that $K=2|B|\left(1+\frac{1}{2^{q-2}-1}\right)$ (cf. the proof of Theorem \ref{Linfty-nonex}) and that $K_q=C_0\max\{2,2^{\frac{2}{2^*_n-q}}\}C_q^{\frac{2}{2^*_n-q}}$, $C_q=C\max\{1,|B|^{\frac{1}{2}\frac{q-2}{2^*_n}}\}$, $C_0=\left(\prod_{k=1}^\infty\frac{2^*_n-q+2}{2}^{k\left(\frac{2}{2^*_n-q+2}\right)^k}\right)^{\frac{1}{2}}$
, and $\gamma_q=\frac{2^*_n-2}{2^*_n-q}$. 
Therefore, if $u$ is non-constant, that is $\int_B \phi^2\,dx\neq 0$, 
 \eqref{eq:limit-q} gives a contradiction for $q$ sufficiently close to $2$.

\item[(ii)] (Assumption on $q$) We remark that the request that $q$ is strictly below the critical exponent $q<(2_n^*+2)/2<2_n^*$, needed in Theorem~\ref{Linfty-nonex}, is exactly the analogue in the fractional setting ($0<s<1$) of the hypothesis in \cite[Theorem 3]{NT}. We expect that this is only a technical assumption due to the specific approach we use. In the local case, in \cite[Proposition 1.4]{LNT} it is proved that the same result holds true for any subcritical $q$. Hence, we expect that it is possible to extend the $L^\infty$-estimate and, consequently, the non-existence result up to the critical exponent also for the nonlocal setting. This would be an interesting point to investigate further.

\item[(iii)] ($L^\infty$-estimates for radial solutions) An analogous $L^{2^*_1}$-$L^\infty$ estimate as in Lemma \ref{lem:L2*Linfty-n}, with the same proof, can be also obtained for solutions $u\in\mathcal C$. In this case, according to Lemma~\ref{immersion}, the critical exponent to be used in \eqref{lower} is $2^*_1$, which is infinity for $s\ge 1/2$. Similarly, a specific $L^\infty$-bound as in Theorem \ref{Linfty-nonex} can be proved for solutions $u\in\mathcal C$, requiring that $q<(2^*_1+2)/2$. 
As we will see in the following section, an $L^\infty$-bound for solutions in $\mathcal C$ will play a crucial role for establishing our existence result.
\end{itemize}
\end{remark}

\section{Uniform $L^\infty$-a priori estimates and existence for radial solutions}\label{S:existence}

In this section we restrict the analysis to radial, radially non-decreasing solutions $u\in\mathcal C$, allowing $q$ to be supercritical in \eqref{P}. Our aim is to extend to the  case $s\le 1/2$ the existence result established in \cite{CC} for the case $s>1/2$. As explained in the Introduction, the idea, originally presented in \cite{BNW}, consists in modifying our nonlinearity (which is possibly supercritical) into a subcritical one, in order to use variational techniques for proving existence. This will be possible thanks to a priori $L^\infty$-estimates for radial, radially non-decreasing solutions which are {\it uniform} in a large class of problems (to which \eqref{P} belongs as well). In the case $s>1/2$ (see \cite{CC}), any function in $\mathcal C$ is automatically $L^\infty(B)$ due to the one-dimensional Sobolev embedding. Here, being $s\le 1/2$, we prove the uniform $L^\infty$-bound making use of a Moser iteration, as done in the previous section.

\subsection{Uniform $L^\infty$-bounds for solutions in $\mathcal C$} 
We first consider the following class of modified problems:
\begin{equation}\label{Pg}
\begin{cases}
d(-\Delta)^s u+ u=g(u)\quad&\mbox{in }B,\\
u\ge 0&\mbox{in }B,\\
\mathcal N_s u=0&\mbox{in }\mathbb R^n\setminus \overline{B},
\end{cases}
\end{equation}
where $g$ can be any function of the form
\begin{equation}\label{eq:g_s0}
g(t)=g_{q,t_0}(t):=\begin{cases}t^{q-1}\quad&\mbox{if }t\in[0,t_0],\\
t_0^{q-1}+\frac{q-1}{\ell-1}t_0^{q-\ell}(t^{\ell-1}-t_0^{\ell-1})&\mbox{if } t\in(t_0,\infty),\end{cases}
\end{equation}
with $t_0\in (1,\infty]$ and $\ell \in (2,\min\{2^*_n,q\})$. In particular, if $t_0=\infty$, $g_{q,\infty}(t)=t^{q-1}$, while, if $t_0<\infty$, $g_{q,t_0}$ is subcritical. Notice that the functions $g$ are of class $C^1$, non-negative, increasing, and satisfy $g(t)\le t^{q-1}$ for every $t\ge 0$ and $g(t)\ge t^{\ell-1}$ for every $t\ge1$.   

We will prove that the solutions belonging to $\mathcal{C}$ of any problem of the class \eqref{Pg} are uniformly bounded in the $L^\infty$-norm, independently of $t_0$ and $\ell$.

\begin{remark}
We emphasize that, in order to modify our nonlinearity into a subcritical one and prove the existence result, it is crucial that the $L^\infty$-bound for solutions of problem \eqref{Pg} does not depend on $t_0$. Indeed, this will allow to choose $t_0$ as in \eqref{choice_t_0} below.
	\end{remark}

We start with a uniform $L^{2^*_1}$-$L^\infty$ estimate.

\begin{lemma}\label{lem:L2*Linfty}
Let $q>2$.

If $s<\frac{1}{2}$, provided that $q < 2^*_1$, there exists a constant $K'_q=K'_q(s,n,q,d)$ depending only on $s$, $n$, $q$, and $d$ such that
\begin{equation}\label{L2*Linfty}
\|u\|_{L^\infty(B)}\le K'_q \left(1+\|u\|^{\gamma'_q}_{L^{2^*_1}(B)}\right)\quad\mbox{ with }\;\gamma'_q:=\frac{2^*_1-2}{2^*_1-q}
\end{equation}
for every $u\in \mathcal C$ solution of \eqref{Pg}.

If $s=\frac{1}{2}$, for every $p>q$ there exists a constant $K'_{q,p}=K'_{q,p}(s,n,q,p,d)$ depending only on $s$, $n$, $q$, $p$, and $d$ such that
\begin{equation}\label{L2*Linfty-1/2}
\|u\|_{L^\infty(B)}\le K'_{q,p} \left(1+\|u\|^{\gamma'_{q,p}}_{L^p(B)}\right)\quad\mbox{ with }\;\gamma'_{q,p}:=\frac{p-2}{p-q}
\end{equation}
for every $u\in \mathcal C$ solution of \eqref{Pg}.
\end{lemma}

\begin{proof} The proof is exactly the same as the one for Lemma \ref{lem:L2*Linfty-n}, with the only difference that the critical exponent here is $2_1^*$ instead of $2_n^*$. In particular, in the chain of inequalities \eqref{upper}, we use that $g_{q,t_0}(u)\le u^{q-1}$ for every $u$, independently of the specific value of $t_0$. This allows to get an estimate which is uniform in $t_0$, as desired.
\end{proof}
\begin{remark}
		Since, differently from the previous section, for the purpose of this section the explicit dependence on the diffusion coefficient $d$ is not important, we have omitted such a dependence in the $L^\infty$-bounds of Lemma \ref{lem:L2*Linfty} above and will do the same in Theorem \ref {thm:Linftybound'} below.
	\end{remark}


\begin{theorem}\label{thm:Linftybound'}
Let $s\in(0,\frac{1}{2}]$ and $q>2$. Assume furthermore that
\begin{equation}\label{eq:extra-cond}
q<\frac{2^*_1+2}{2}\quad\mbox{ if }\quad s<\frac{1}{2}.
\end{equation} 
Then, there exists a constant $K'_\infty=K'_\infty(s,n,q,d)$ depending only on $s$, $n$, $q$, and $d$ such that 
\[
\|u\|_{L^\infty(B)}\le K'_\infty \quad\mbox{ for every $u\in\mathcal C$ solution of \eqref{Pg}.}
\]
\end{theorem}

\begin{proof}
As for Theorem \ref{Linfty-nonex}, we integrate the equation in \eqref{Pg} to get
\begin{equation*}
	\begin{split}
		\int_B u \,dx& =\int_B g(u)\,dx = \int_{B\cap \{u\le 2\}} g(u)\,dx + \int_{B\cap \{u>2\}} g(u)\,dx \\
		&\ge (1+\delta')\int_{B\cap \{u>2\}} u\,dx,
		\end{split}
\end{equation*}
where $\delta':=2^{\ell-2}-1>0$ (independent of $t_0$), being $g(t)\ge t^{\ell-1}$ for $t>2$. Hence, we obtain
\begin{equation}\label{g}
\int_B g(u)\,dx=\int_B u\,dx \le 2|B|\left(1+\frac{1}{\delta'}\right)=:K'.
\end{equation}
Testing the equation in \eqref{Pg} with $u$, using the embedding in Lemma \ref{immersion}, and that $u\in L^\infty(B)$ by Lemma \ref{lem:L2*Linfty}, we obtain for $s<\frac{1}{2}$
\begin{equation*}
\|u\|_{L^{2^*_1}(B)}^2\le C^2 \|u\|^2_{H^s_{B,0}} {\le} C^2{\max\left\{1,\frac{1}{d}\right\}} \int_B g(u)u\,dx
\le C^2 \max\left\{1,\frac{1}{d}\right\} K' \|u\|_{L^\infty(B)}. 
\end{equation*}
Hence, by \eqref{L2*Linfty}, 
\begin{equation*}
\begin{split}
\|u\|_{L^\infty(B)}&\le K'_q\left(1+\|u\|_{L^{2^*_1}(B)}^{\gamma'_q}\right)\\
&\le K'_q\left[ 1+C^{\gamma'_q} \left(\max\left\{1,\frac{1}{d}\right\} K'\right)^{\frac{\gamma'_q}{2}} \|u\|_{L^\infty(B)}^{\frac{\gamma'_q}{2}}\right]. 
\end{split}
\end{equation*}
Since, by \eqref{eq:extra-cond}, $\frac{\gamma'_q}{2}<1$, the existence of a constant $K'_\infty$ as in the statement, follows. The conclusion for $s=1/2$ follows analogously using \eqref{L2*Linfty-1/2} and observing that $\frac{\gamma'_{q,p}}{2}<1$ for $p$ large. 
\end{proof}


\subsection{Existence of a non-constant solution in $\mathcal C$}
We follow the lines of \cite{BNW,CC}. Since, by Theorem \ref{thm:Linftybound'}, $K'_\infty$ is independent of $t_0$, we can choose $g=g_{q,t_0}$ with $t_0$ satisfying
\begin{equation}\label{choice_t_0}
	K'_\infty+1<t_0<\infty.
	\end{equation} 
With this choice, problem \eqref{Pg} is subcritical and any solution of \eqref{Pg} is also a solution of \eqref{P}. 

We introduce the energy functional associated to problem \eqref{Pg}, namely $\mathcal E:H^s_{B,0}\to \mathbb R$ such that for every $u\in H^s_{B,0}$,
\[
\mathcal E(u):=d\frac{c_{n,s}}{4}\iint_{\mathbb R^{2n}\setminus(B^c)^2} \frac{|u(x)-u(y)|^2}{|x-y|^{n+2s}}\,dx\,dy+\frac{1}{2}\int_B u^2\,dx-\int_B G(u)\,dx,
\]
with $G(t):=\int_0^t g(\tau)\,d\tau$. In view of the subcritical growth of $g$, $\mathcal{E}$ is well defined and of class $C^2$ in $H^s_{B,0}$. Moreover, its critical points are weak solutions of \eqref{Pg}. 

Now, by the subcritical growth and the structure of $\mathcal E$, it is possible to prove that the functional $\mathcal E$ satisfies the compactness and geometry assumptions that allow to find a mountain pass-type critical point. The following lemmas continue to hold also for $0<s\le 1/2$, with the same proofs given in \cite[Section 4]{CC} for the case $s>1/2$.

The next lemma states that the functional $\mathcal E$ satisfies the Palais-Smale condition. 

\begin{lemma}[{\cite[Lemma 4.2 and Corollary 4.4]{CC}}]\label{lem:PS}
Let $(u_k)\subset H^s_{B,0}$ be such that $(\mathcal E(u_k))$ is bounded and $\mathcal E'(u_k)\to 0$ in $(H^s_{B,0})^*$, then $(u_k)$ admits a convergent subsequence. As a consequence, let $c\in\mathbb R$ be such that $\mathcal E'(u)\neq 0$ for every $u\in\mathcal C$ having $\mathcal E(u)=c$, then there exist $\bar\varepsilon,\,\bar\delta>0$ such that $\|\mathcal E'(u)\|_*\ge \bar\delta$ for every $u\in\mathcal C$ with $|\mathcal E(u)-c|\le 2\bar\varepsilon$.
\end{lemma}

Furthermore, the following refined version of the Deformation Lemma holds. 
 
\begin{lemma}[{\cite[Lemma 4.8]{CC}}]\label{lem:DF}
Let $c\in\mathbb R$ be such that $\mathcal E'(u)\neq 0$ for every $u\in\mathcal C$ having $\mathcal E(u)=c$. Then, there exists a continuous function $\eta:\mathcal C\to\mathcal C$ such that 
\begin{itemize}
\item[(i)] $\mathcal E(\eta(u))\le \mathcal E(u)$ for all $u\in\mathcal C$;
\item[(ii)] $\mathcal E(\eta(u))\le c-\bar\varepsilon$ for all $u\in\mathcal C$ such that $|\mathcal E(u)-c|<\bar\varepsilon$;
\item[(iii)] $\eta(u)=u$ for all $u\in\mathcal C$ such that $|\mathcal E(u)-c|>2\bar\varepsilon$,
\end{itemize}
where $\bar\varepsilon$ is the constant corresponding to $c$ given in Lemma \ref{lem:PS}.
\end{lemma}
 
Finally, the functional $\mathcal E$ has the mountain pass geometry. 

\begin{lemma}[{\cite[Lemma 4.9]{CC}}]\label{lem:MPG} 
Let $\tau\in(0,1)$, there exists $\alpha>0$ such that 
\begin{itemize}
\item[(i)] $\mathcal E(u)\ge \alpha$ for every $u\in\mathcal C$ with $\|u\|_{L^\infty(B)}=\tau$;
\item[(ii)] there exists $\bar u\in\mathcal C$ with $\|\bar u\|_{L^\infty(B)}>\tau$ such that $\mathcal E(\bar u)<0$.
\end{itemize}
\end{lemma}

In correspondence of $\tau$ and $\alpha$ as in Lemma \ref{lem:MPG}, we put
\[
\begin{gathered}
U_0:=\left\{u\in\mathcal C\,:\,\mathcal E(u)<\frac{\alpha}{2},\,\|u\|_{L^\infty(B)}<\tau\right\},\\
U_\infty:=\left\{u\in\mathcal C\,:\,\mathcal E(u)<0,\,\|u\|_{L^\infty(B)}>\tau\right\}.
\end{gathered}
\]

This allows to conclude the existence part. 

\begin{proposition}[{\cite[Proposition 4.12]{CC}}]\label{prop:MP}
The value 
\[
c:=\inf_{\gamma\in\Gamma}\max_{t\in[0,1]}\mathcal E(\gamma(u)), \mbox{ with } \Gamma:=\left\{\gamma\in C([0,1];\mathcal C)\,:\,\gamma(0)\in U_0,\,\gamma(1)\in U_\infty\right\}
\]
is positive and finite, and there exists a critical point $u\in\mathcal C$ with $\mathcal E(u)=c$. Moreover, $u$ solves the original problem \eqref{P}.\end{proposition}
\begin{proof}
In view of Lemmas \ref{lem:PS}, \ref{lem:DF}, and \ref{lem:MPG}, the existence of a critical point $u$ at the minimax level $c$ is standard. Hence, $u$ solves problem \eqref{Pg}. Now, as already anticipated at the beginning of the present subsection, by Theorem~\ref{thm:Linftybound'}, $\|u\|_{L^\infty(B)}\le K'_\infty$ and so, by the choice of $t_0$, which is greater than $K'_\infty+1$ by \eqref{choice_t_0}, $g(u)=u^{q-1}$ in $B$. This implies that $u$ solves also \eqref{P}. 
\end{proof}

\begin{proof}[Proof of Theorem \ref{thm:main}]
The existence of a solution $u\in\mathcal C$ of \eqref{P} has been proved in Proposition \ref{prop:MP}. It remains to prove that $u$ is not constant. Problem \eqref{P} has only two constant solutions, 0 and 1, and, from its energy level, we know that $u\not\equiv 0$. Let us show now that $u\not\equiv 1$. This is done performing a local analysis around the constant solution 1: we prove that in every neighbourhood of 1 we can find functions having lower energy. This allows us to construct an admissible path $\gamma\in\Gamma$ whose points $\gamma(t)\in\mathcal C$ have all energy below the energy of 1. 
The proof can be done exactly as in \cite[Lemmas 5.2 and 5.3, Theorem 1.1]{CC}. We report here only a sketch of the proof to illustrate the main lines, highlighting the importance of the hypothesis 
\begin{equation}\label{eq:hypo-q}
q>2+d\lambda_{2,\mathrm{r}}^+,
\end{equation} 
and propose a more direct conclusion that exploits the specific form of the nonlinearity in \eqref{P}.
We perturb the constant 1 in the direction of the second radial and radially increasing eigenfunction $\varphi_2$ that corresponds to the eigenvalue $\lambda_{2,\mathrm{r}}^+$. Via a second-order Taylor expansion of $\mathcal E$ centered at 1, we prove that the energy of functions of the form $v_t:=(1+o(t))(1+t\varphi_2)$ satisfies for $t\in(-\varepsilon,\varepsilon)$ and $\varepsilon>0$ small,
\begin{equation}\label{eq:comp-en}
\begin{aligned}
\mathcal E(v_t)-\mathcal E(1)&=\frac{1}{2}\mathcal E''(1)[t\varphi_2+o(t),t\varphi_2+o(t)]+o(t^2)\\&=\frac{t^2}{2}\left(d[\varphi_2]^2_{H^s_{B,0}}-\int_B(q-2)\varphi_2^2\,dx\right)+o(t^2)\\
&=\frac{t^2}{2}(d\lambda_{2,\mathrm{r}}^+-q+2)\int_B\varphi_2^2\,dx+o(t^2)<0,
\end{aligned}
\end{equation}
where, in the last step, we have taken into account \eqref{eq:hypo-q}. 

Now, let $\tau$ be given as in Lemma \ref{lem:MPG} and fix $t_\infty>1$ so large that the constant function $t_\infty$ belongs to $U_\infty$. By continuity, we can choose $\bar t\in(0,\varepsilon)$ so small that $1+\bar t\varphi_2\in\mathcal C$ and $t_\infty(1+\bar t \varphi_2)\in U_\infty$. Hence, the path defined as $\bar \gamma(t)=tt_\infty(1+\bar t\varphi_2)$ for every $t\in[0,1]$ belongs to the set $\Gamma$ of admissible paths. 
Moreover, via explicit calculations, it is possible to prove that $\mathcal E(\bar\gamma(\cdot))$ attains its unique maximum at some $t'\in(0,1)$ such that 
\[
\begin{aligned}
t't_\infty&=\left(\frac{d[1+\bar t\varphi_2]^2_{H^s_{B,0}}+\|1+\bar t \varphi_2\|^2_{L^2(B)}}{\|1+\bar t \varphi_2\|^q_{L^q(B)}}\right)^{\frac{1}{q-2}}\\
&=\left(\frac{d\bar{t}^2[\varphi_2]^2_{H^s_{B,0}}+|B|+\bar{t}^2 \int_B\varphi_2^2\,dx}{\int_B(1+\bar t \varphi_2)^q\,dx}\right)^{\frac{1}{q-2}}\\
&=\left(\frac{|B|+o(\bar{t})}{|B|+o(\bar t)}\right)^{\frac{1}{q-2}}=1+o(\bar{t})\quad\mbox{as }\bar t\to 0,
\end{aligned}
\]
where we have taken into account that $\int_B\varphi_2\,dx=0$. 
Hence, the point of $\bar \gamma$ having maximal energy can be written in the form $\bar\gamma(t')=v_{\bar{t}}=(1+o(\bar{t}))(1+\bar{t}\varphi_2)$, and so, by \eqref{eq:comp-en}, 
\[
\mathcal E(u)=\inf_{\gamma\in\Gamma}\max_{t\in[0,1]}\mathcal E(\gamma(t))\le \max_{t\in[0,1]}\mathcal E(\bar\gamma(t))=\mathcal E(v_{\bar{t}})<\mathcal E(1).
\]
In particular, $u\not\equiv 1$.

Finally, by the maximum principle \cite[Theorem 2.6]{CC} (see Remark (ii) below, for the regularity properties needed on $u$ for the validity of such maximum principle), $u$ is positive almost everywhere in $B$. This, combined with the symmetry and monotonicity of $u$, implies that the solution can vanish only at the origin. 
\end{proof}

\begin{remark}
	Let us conclude with two remarks.
	\begin{itemize}
\item[(i)] (Regularity of the solution) We observe that being $u\in L^\infty(B)$ by Theorem~\ref{thm:Linftybound'}, we get by nonlocal Neumann conditions that $u\in L^\infty(\mathbb R^n)$, see the proof of \cite[Lemma 3.6]{CC}\footnote{We warn the reader that Lemma 3.6 of \cite{CC} is not stated correctly. Indeed, as it can be seen from the proof, the solution $u$ is of class $C^2(B)\cap L^\infty(\mathbb R^n)$ and not of class $C^2(\mathbb R^n)$.}.  
Moreover, using \cite[Proposition 2.9]{Silvestre} with $w=u^{q-1}-u\in L^\infty(\mathbb R^n)$, we obtain $u\in C^{0,\alpha}(\mathbb R^n)$ for every $\alpha\in (0,2s)$. Then, we can use a bootstrap argument and apply \cite[Proposition 2.8]{Silvestre} to conclude that $u$ has the following regularity: $u\in C^{2}(B)$ if $q>3-2s$, and $u\in C^{1,q-2+2s}(B)$ if $2<q\le 3-2s$.
\item[(ii)] (Maximum principle) 
The maximum principle established in \cite[Theorem 2.6]{CC} is stated for ``classical" $s$-superharmonic functions having non-negative nonlocal Neumann conditions, namely for functions $u$ satisfying $(-\Delta)^s u \ge 0$ in a pointwise (and not just weak) sense. In order to be able to write $(-\Delta)^s u$ almost everywhere, some further regularity assumptions are required on $u\in H^s_{B,0}$: if $s>1/2$, it is needed to take $u\in C^{1,2s+\varepsilon-1}(B)$, while if $s\le\frac{1}{2}$, it is enough to assume $u\in C^{2s+\varepsilon}(B)$ for some $\varepsilon\in(0,1)$. Since in \cite{CC}, we were in the case $s>1/2$, we needed to require $u\in C^{1,2s+\varepsilon-1}(B)$ and we stated the maximum principle under this assumption. In the present paper, being $s\le\frac{1}{2}$, the requirement needed the for the validity of such maximum principle is $u\in C^{2s+\varepsilon}(B)$.
\end{itemize} 
\end{remark}
\footnotesize{
\noindent \textbf{Acknowledgments.}
E.C. was partially supported by the INdAM-GNAMPA Project 2022 {``Problemi con Aspetti 
		Competizione per Energie 
		spettrali ''}, by grants MTM2017-84214-C2-1-P and RED2018-102650-T funded by MCIN/AEI/10.13039/501100011033 and by ``ERDF A way of making Europe'', and
	is part of the Catalan research group 2017 SGR 1392. F.C. was partially supported by the INdAM-GNAMPA Project 2022 ``Studi asintotici in problemi parabolici ed ellittici ''. 
		
		The authors are grateful to Lorenzo Brasco and to Antonio J. Fern\'andez  for some useful discussions
		on the subject of this paper.
		
\noindent \textbf{Author Contribution.} All authors wrote and reviewed the manuscript.

\noindent \textbf{Conflict of interests.} The authors have no competing interests as defined by Springer, or other interests that might be perceived to influence the results and/or discussion reported in this paper.
}
\bibliographystyle{abbrv}

\end{document}